\documentclass[11pt,reqno]{amsart}

\usepackage[left=1in, right=1in, top=1.3in, bottom=1.2in]{geometry}

\usepackage{amssymb,amsthm,amsmath}
\usepackage{hyperref}
\usepackage{cleveref}

\newtheorem{theorem}{Theorem}[section]
\newtheorem{definition}[theorem]{Definition}
\newtheorem{lemma}[theorem]{Lemma}

\newtheorem{corollary}[theorem]{Corollary}
\newtheorem{remark}[theorem]{Remark}

\hypersetup{colorlinks=true, linkcolor=black, filecolor=black, urlcolor=black}

\begin{document}

\newcommand\bnode[3][2]{\node (#2) at (0, 0) {}; \fill (0, 0) circle (1pt) node[#1] {#3};}

\title{Spectral radii of sparse non-Hermitian random matrices}
\author[Hyungwon Han]{Hyungwon Han}
\address{Department of Mathematical Sciences, KAIST, South Korea}
\email{measure@kaist.ac.kr}

\begin{abstract}
We provide estimates for the spectral radii of an $n\times n$ sparse non-Hermitian random matrix $Z$ with general entries in the regime $p=d/n$ where $0<d<1$ is fixed. Utilizing the structural results of \cite{l90}, we show that the spectral radius $\rho (Z)$ is $0$ with probability converging to some nonzero value, and satisfies the inequality $(\phi (n))^{-1}\leq \rho (Z)\leq \phi (n)$ in the asymptotic sense for any function $\phi$ satisfying $\lim_{n\to\infty}\phi (n)=\infty$ with the remaining probability.
\end{abstract} 

\maketitle


\let\thefootnote\relax

\bigskip

\section{Introduction}

Spectral radii of non-Hermitian random matrices has been a topic of great interest over the last few decades. However, due to spectral instability, its study is considered challenging compared to the Hermitian counterpart, and the asymptotic behavior for dense non-Hermitian matrices was completely established only recently, evolving through the works of \cite{ge86, bcct, bcg}. For sparse non-Hermitian matrices, even the law of large numbers behavior of the spectral radii is incomplete. While an asymptotic value of $\sqrt{np}$ can be deduced in the regime $p= n^{-1+\varepsilon}$ for $\varepsilon >0$ by modifying the works on dense matrices such as \cite{bcct}, the only result available for $p$ near the critical value $1/n$ is \cite{bbk20}. Although the authors provide precise estimates for the eigenvalues of unweighted Erd\H{o}s-R\'enyi digraphs, much less is obtained for networks with general weights especially when $p=d/n$ as we discuss in Section 3.1.

The analogous problem for the largest eigenvalue of Hermitian random matrices has been settled even in the sparse setting. For the regime $p\gg 1/n$, it was shown that the largest eigenvalue is asymptotically $2\sqrt{np}$, for instance, in the works of \cite{kh01, ty21}. Recently, the behavior in the regime $p=d/n$ for fixed $d>0$ was established for unweighted Erd\"os-R\'enyi graphs in \cite{bbg21} and for weighted networks in \cite{gn22, ghn22}.

Most of the previous results on the spectral radii of non-Hermitian random matrices --- both sparse and dense --- are focused on sharpening the moment estimates, which are inapplicable near the critical value $p=1/n$. In this paper, we use structural results on Erd\H{o}s-R\'enyi digraphs from \cite{l90} to address the regime $p=d/n$ where $0<d<1$ is a fixed constant. Roughly speaking, we show that the spectral radius is $0$ with some nonzero probability, and approximately of order $1$, or to be exact, $\Theta_{\mathsf{p}}(1)$ with the remaining probability, which we define in Definition \ref{def:asymptotic}.

Let $X$, $Y$ be $n\times n$ i.i.d. matrices where $X$ has $\mathrm{Bernoulli}(p)$ entries with $p=p_n$ depending on $n$, and $Y$ has entries $Y_{ij}$ satisfying $Y_{ij}=0$ and $\mathbb{E}|Y_{ij}|^2=1$. Furthermore, we assume that $Y_{ij}$ is not concentrated at $0$, i.e., $\mathbb{P}(Y_{ij}=0)=0$, which is necessary for the establishment of a lower bound in Theorem \ref{constant lower bound}. Define $Z=X\odot Y$ as the entrywise product $Z_{ij}=X_{ij}Y_{ij}$ for $1\leq i,j\leq n$. 
We mainly focus on the constant average degree regime where $p=d/n$ with fixed $0<d<1$. 

The underlying Erd\H{o}s-R\'enyi digraph $\mathcal{G}_{\mathsf{d}}(n,d/n)$ exhibits a phase transition at $d=1$.
If $d<1$, $\mathcal{G}_{\mathsf{d}}(n,d/n)$ has a simple structure where all nontrivial strongly connected components are directed cycles \cite{l90}. We utilize this characterization to show the upper bound result in Theorem \ref{constant theorem} below.
To state our results precisely, we first adapt the following asymptotic notations from \cite{ls09}. We say that an event $\mathcal{A}_n$ depending on $n$ holds \textit{with high probability} if $\mathbb{P}(\mathcal{A}_n)\to 1$ as $n\to\infty$.

\begin{definition} \label{def:asymptotic}
\upshape{
For a sequence of random variables $\lbrace X_n\rbrace$ we write $X_n=O_{\mathsf{p}}(f(n))$ ($X_n=\Omega_{\mathsf{p}}(f(n))$) if for every function $\phi$ satisfying $\lim_{n\to\infty}\phi (n)\to\infty$, the inequality $X_n\leq \phi (n)f(n)$ ($X_n\geq f(n)/\phi (n)$) holds with high probability. We will write $X_n=\Theta_{\mathsf{p}} (f(n))$ if both $X_n=O_{\mathsf{p}}(f(n))$ and $X_n=\Omega_{\mathsf{p}}(f(n))$ hold.
}
\end{definition}

\begin{theorem} \label{constant theorem}
Suppose $p=d/n$ where $0<d<1$ is fixed and $Z$ satisfies the conditions above. Then $\rho (Z) =O_{\mathsf{p}}(1)$.
\end{theorem}

We may deduce an analogous lower bound again depending on the cycle structure of $\mathcal{G}_{\mathsf{d}}(n,d/n)$. If the underlying graph is \textit{acyclic}, $Z$ is nilpotent, thus $\rho (Z)=0$. In contrast, the existence of a cycle allows us to place suitable weights on it and thus provide a lower bound for a specific entry of $Z^k$ for large $k$ depending on $n$. Using Theorem \ref{d<1 acyclic}, we have the following result:

\begin{theorem} \label{constant lower bound}
For $p=d/n$ where $0<d<1$ is fixed, 
$\mathcal{G}_{\mathsf{d}}(n,p)$ is acyclic and thus $\rho (Z)=0$ with probability converging to $(1-d)e^{d+d^2/2}$. Conditioned on the event that $\mathcal{G}_{\mathsf{d}}(n,p)$ admits a cycle, with high probability, $\rho (Z)=\Omega_{\mathsf{p}}(1)$.
\end{theorem}

Combining these bounds we have the following main result.

\begin{corollary} \label{main}
Suppose $Z$ satisfies the conditions of Theorem \ref{constant theorem} and \ref{constant lower bound}. Then $\rho (Z)=0$ with probability converging to $(1-d)e^{d+d^2/2}$, and with the remaining probability, $\rho (Z)=\Theta_p (1)$.
\end{corollary}

Next we state some minor results in the supercritical and subcritical regimes that can be easily deduced from the previous works of \cite{bbk20} and \cite{ks03}.

\subsection{Supercritical regime, $p\gg 1/n$}

For the regime $p\gg 1/n$, the lower bound for the spectral radius is readily provided by the sparse circular law in \cite{rt19}. For $p$ sufficiently larger than $1/n$, the estimates from \cite{bbk20} with truncation yields a matching upper bound. We consider the entries $Y_{ij}$ having a Weibull distribution with shape parameter $\alpha$, i.e.,
\begin{align*}
C_1e^{-t^{\alpha}}\leq \mathbb{P}(|Y_{ij}|\geq t) \leq C_2 e^{-t^{\alpha}}
\end{align*}
holds for some constants $C_1,C_2>0$ and all $t>1$.
For such entries, the regime where this method is applicable depends on the shape parameter $\alpha$. 

\begin{theorem} \label{polynomial theorem}
Suppose $Z$ has Weibull entries with shape parameter $\alpha >0$ and $p_n\gg (\log n)^{\frac{2}{\alpha}}/n$. Then 
\begin{align*}
\frac{\rho (Z)}{\sqrt{np_n}}\to 1
\end{align*}
in probability.
\end{theorem}

\subsection{Subcritical regime, $p\ll 1/n$}

In this case $\mathcal{G}_{\mathsf{d}}(n,p)$ is a forest with high probability (Lemma \ref{forest}), and $Z$ is nilpotent. The following theorem is readily deduced from this fact.

\begin{theorem} \label{subcritical theorem}
Suppose $Z$ has arbitrary entries and $np_n\to 0$ as $n\to\infty$. Then $\rho (Z)=0$ holds with high probability.
\end{theorem}

\begin{remark} \upshape{
In the case of Hermitian matrices with $p=d/n$, the spectral behavior of the largest eigenvalue analyzed in \cite{gn22, ghn22} was independent of the specific value of $d$. However, Corollary \ref{main} heavily depends on $d$ and our proofs cannot be extended to the case where $d\geq 1$.

For $\mathcal{G}_{\mathsf{d}}(n,d/n)$ where $d>0$ is fixed, it was recently proved in \cite[Theorem 1.1]{sss23.2} that the empirical spectral measure converges.
The authors also show that the limiting measure is  $\delta_0$, i.e., the Dirac mass at $0$ for $0<d\leq 1$, and a nontrivial distribution for $d>1$, where \cite[Theorem 1]{l90} (Theorem \ref{digraph} in this paper) is used for proving the first result. It seems that unlike its Hermitian analogues, the spectra of sparse non-Hermitian matrices in the critical regime are directly related with the phase transition of the underlying Erd\H{o}s-R\'enyi digraphs.
}
\end{remark}

\subsection{Notations}

The spectral radius of a matrix $A$ will be denoted $\rho (A)$.
We denote the Erd\H{o}s-R\'enyi graph (digraph) with $n$ vertices and edge probability $p$ by $\mathcal{G}(n,p)$ ($\mathcal{G}_{\mathsf{d}}(n,p)$).
The \textit{strongly connected components} of a digraph are the maximal subgraphs in which every pair of vertices can be connected by a directed path.

Consider two functions $f,g:\mathbb{N}\to [0,\infty )$. We write $f\ll g$ if $\lim_{n\to\infty}f(n)/g(n) =0$ and $f\lesssim g$ if $\lim_{n\to\infty}f(n)/g(n)<\infty$. We denote $f\approx g$ if both $f\lesssim g$ and $g\lesssim f$ hold.

\subsection{Organization of paper}

In Section 2 we provide some background results such as the sparse circular law, structural results of $\mathcal{G}_{\mathsf{d}}(n,d/n)$, and Gelfand's formula. Section 3 is devoted to the proof of upper and lower bound estimates in Theorems \ref{constant theorem} and \ref{constant lower bound} using the structural results of $\mathcal{G}_{\mathsf{d}}(n,d/n)$ in Theorem \ref{digraph}. Finally Section 4 contains a simple proof for Theorem \ref{polynomial theorem} by applying Theorem \ref{upper bound} from \cite{bbk20}.

\subsection{Acknowledgements}

The author thanks Kyeongsik Nam for pointing out helpful references.

\section{Backgrounds}

This section is devoted to some background results regarding the properties of sparse non-Hermitian random matrices and the structural analysis of Erd\H{o}s-R\'enyi digraphs which are crucial for the proof of our main theorems. 

\subsection{Upper bound results}

Suppose that an i.i.d. matrix $A=(A_{ij})_{1\leq i,j\leq n}$ has mean zero entries and satisfies the following conditions:
\begin{align} \label{conditions}
\mathbb{E}|A_{ij}|^2 \leq \frac{1}{n}, \quad \max_{1\leq i,j\leq n}|A_{ij}| \leq \frac{1}{q}\  \text{ a.s.}
\end{align}
Then the following upper tail estimate for $\rho (A)$ holds:
\begin{theorem}[{\cite[Theorem 2.11]{bbk20}}] \label{upper bound}
There exist universal constants $C,c>0$ such that for $q\leq n^{1/10}$ and $\varepsilon \geq 0$,
\begin{align*}
\mathbb{P}\left( \rho (A) \geq 1+\varepsilon \right) \leq Cn^{2-cq\log (1+\varepsilon )}.
\end{align*}
\end{theorem}

\subsection{Lower bound results: sparse circular law}

The sparse version of the circular law was first proven in \cite{br19} for $p\gg \frac{\log^2 n}{n}$ and generalized to $p\gg \frac{1}{n}$ in \cite{rt19}. Recently, a simpler proof was given by \cite{sss23} yielding the same result.

\begin{theorem}[{\cite[Theorem 1.2]{rt19}}] \label{sparse circular law}
Let $A$ be an $n\times n$ i.i.d. matrix with entries $A_{ij}=B_{ij}Y_{ij}$ where $B_{ij}\sim \mathrm{Bernoulli}(p_n)$ and $Y_{ij}$ has zero mean and unit variance. If $\lim_{n\to\infty}np_n=\infty$, then the empirical spectral distribution of $\frac{1}{\sqrt{np_n}}A$ converges weakly in probability to the uniform measure on the unit disc of the complex plane.
\end{theorem}

\subsection{Structure of $\mathcal{G}_{\mathsf{d}}(n,d/n)$}

The Erd\H{o}s-R\'enyi digraph $\mathcal{G}_{\mathsf{d}}(n,p)$ is a random digraph with $n$ vertices and each directed edge chosen with independent probability $p$. With high probability, we may assume that $\mathcal{G}_{\mathsf{d}}(n,d/n)$ has no self-loops. Then the following theorem characterizes the structure of $\mathcal{G}_{\mathsf{d}}(n,d/n)$ in terms of directed cycles with relatively short length.

\begin{theorem}[{\cite[Theorem 1]{l90}}] \label{digraph}
Let $d>0$ be a positive constant and $\omega$ be a function such that $\omega (n)\to \infty$ as $n\to\infty$.
\begin{enumerate}
\item If $d<1$, then with high probability, each nontrivial (i.e., containing more than one vertex) strongly connected component of $\mathcal{G}_{\mathsf{d}}(n,d/n)$ is a directed cycle with length smaller than $\omega (n)$.
\item If $d>1$, there exists a positive constant $\alpha =\alpha (d)$ such that with high probability, $\mathcal{G}_{\mathsf{d}}(n,d/n)$ contains a strongly connected component larger than $\alpha n$, and every other nontrivial strongly connected component is a directed cycle with length smaller than $\omega (n)$.
\end{enumerate}
\end{theorem}

For $d<1$, part (1) provides a simple characterization of all strongly connected components, which is crucial for the proof of \ref{constant theorem}. In contrast, by part (2), a large complex component arises for $d>1$ which makes our methods inapplicable. 

When $d<1$, the probability that $\mathcal{G}_{\mathsf{d}}(n,d/n)$ is acyclic converges to a nonzero number as $n\to\infty$, and this limit can be explicitly calculated:

\begin{theorem}[{\cite[Theorem 2]{rrw20}}] \label{d<1 acyclic}
Let $\mathbb{A}(n,p)$ denote the probability that $\mathcal{G}_{\mathsf{d}}(n,p)$ is acyclic. For $p=d/n$ with fixed $d<1$,
\begin{align*}
\lim_{n\to\infty}\mathbb{A}(n,p) =(1-d)e^{d+d^2/2}.
\end{align*}
\end{theorem}

Finally we state a lemma on the structure of $\mathcal{G}(n,d/n)$.

\begin{lemma}[{\cite[Lemma 2.1 (ii)]{ks03}}] \label{forest}
If $np\to 0$, then $\mathcal{G}(n,p)$ is a forest with high probability.
\end{lemma}

\subsection{Gelfand's formula}

Gelfand's formula states that for the spectral radius $\rho (\cdot )$ and any matrix norm $\| \cdot \|$, 
\begin{align}
\rho (A) =\lim_{k\to\infty}\| A^k\|^{1/k}.
\end{align}
Furthermore, the quantities $\| A^k\|^{1/k}$ decreases to $\rho (A)$ when $\| \cdot \|$ is a \textit{consistent} matrix norm, which includes the spectral norm
\begin{align}
\| A\|_{2\to 2} := \sup_{x\in \mathbb{C}^n, \| x\|_2=1}\| Ax\|_2
\end{align}
that we will mainly use. From now on we will simply denote the spectral norm by $\| \cdot \|$. For a matrix $A$ with $\rho (A)\neq 0$, there is a corresponding lower bound
\begin{align}
\gamma^{(1+\log k)/k}\| A^k\|^{1/k} \leq \rho (A) \leq \| A^k\|^{1/k}
\end{align}
with some constant $\gamma \in (0,1)$. The following theorem provides a more explicit lower bound which plays a key role in proving Theorem \ref{constant lower bound}.
\begin{theorem}[{\cite[Theorem 1]{k09}}] \label{rho bound}
For an $n\times n$ matrix $A$ with $n>2$ and for $k\geq 1$,
\begin{align*}
C_n^{-\sigma_n(k)/k}\left( \frac{\| A\|^n}{\| A^n\|}\right)^{-\nu_n(k)/k}\| A^k\|^{1/k}\leq \rho (A)\leq \| A^k\|^{1/k},
\end{align*}
where
\begin{align*}
C_n =n^{3n/2}, \quad \sigma_n(k) = \frac{(n-1)^3}{(n-2)^2}k^{\frac{\log (n-1)}{\log n}}, \quad \nu_n(k) =\frac{(n-1)^2}{n-2}k^{\frac{\log (n-1)}{\log n}}.
\end{align*}
\end{theorem}

\section{Proof of Theorems \ref{constant theorem} and \ref{constant lower bound}}

In this section we prove Theorems \ref{constant theorem} and \ref{constant lower bound}. We start with a brief examination that the upper bound from direct application of Theorem \ref{upper bound} is insufficient for our results.

\subsection{Weak upper bound}

Let $X_{ij}\sim \mathrm{Bernoulli}(d/n)$ where $d>0$ is fixed and $Y_{ij}$ be a Weibull distribution with shape parameter $\alpha >0$ satisfying $\mathbb{E}|Y_{ij}|^2=1$.
We consider the i.i.d. matrix $H$ with entries
\begin{align*}
H_{ij}:= \frac{1}{(\log n)^{\frac{1}{\alpha}}}X_{ij}Y_{ij}.
\end{align*}
Clearly $\mathbb{E}|H_{ij}|^2 \leq \frac{1}{n}$. For any constant $t>1$, 
\begin{align*}
\mathbb{P}\left( \max_{1\leq i,j\leq n}|H_{ij}|> t\right) & \leq \sum_{1\leq i,j\leq n}\mathbb{P}(|H_{ij}|>t) =n^2 \mathbb{P}(|Y_{ij}| >t\log^{\frac{1}{\alpha}} n)\mathbb{P}(X_{ij}=1) \nonumber\\
& \lesssim dne^{-t^{\alpha}\log n} =o(1).
\end{align*}
Now we condition on the high probability event
\begin{align*}
\mathcal{A}_{q,n} :=\left\lbrace \max_{1\leq i,j\leq n}|H_{ij}|\leq \frac{1}{q}\right\rbrace 
\end{align*}
for some $0<q<1$. Applying Theorem \ref{upper bound} to $H$ yields
\begin{align*}
\mathbb{P}(\rho (H) \geq 1+\varepsilon )\leq Cn^{2-cq\log (1+\varepsilon )} =o(1)
\end{align*}
conditioned on $\mathcal{A}_{q,n}$, for sufficiently large $\varepsilon >0$ depending on $c$ and $q$. Thus we have $\rho (Z) =O(\log^{\frac{1}{\alpha}}n)$ with high probability, which is far weaker than Theorem \ref{constant theorem}.
Also note that the estimate depends heavily on the Weibull tail condition. Much weaker bounds will be obtained if we only assume finite moment conditions on $Y_{ij}$, for instance, an $O(\sqrt{n})$ bound when only assuming $\mathbb{E}|Y_{ij}|^2=1$.

\subsection{Proof of Theorem \ref{constant theorem}}

Consider $X\sim \mathcal{G}_{\mathsf{d}} (n,d/n)$ and $Z=X\odot Y$. Let $\mathcal{C}_{g(n)}$ be the event where all nontrivial strongly connected components of $X$ are directed cycles with size smaller than $g(n)$. For simplicity, we now use the term \textit{cycle} to mean directed cycles, except when the directedness must be emphasized. Condition on the high probability event $\mathcal{C}_{\omega (n)}$ for some function $\omega$ such that $\lim_{n\to\infty}\omega (n)=\infty$ to be determined later. From the proof of \cite[Theorem 1]{l90}, 
\begin{align}
\mathbb{P}(\mathcal{C}_{\omega (n)}^c) \leq \frac{d^{\omega (n)}}{1-d}.
\end{align} 
Let $N$ be the number of cycles with size greater than or equal to $m$. Again using the proof of \cite[Theorem 1]{l90},
\begin{align}
\mathbb{P}(N>\eta (n)) \leq \frac{1}{\eta (n)}\mathbb{E}N \leq \frac{1}{\eta (n)}\frac{d^m}{1-d}\leq \frac{1}{1-d}\frac{1}{\eta (n)}.
\end{align}
Then we may condition on the high probability event $\lbrace N \leq \eta (n)\rbrace$ for a function $\eta$ satisfying $\lim_{n\to\infty}\eta (n)=\infty$. This implies that there are at most $\eta (n)$ cycles with at most $\eta (n) \omega (n)$ edges in them.

Now we estimate the entries of $Z^{k}$, which we denote by $W\odot A$ where the entries of $W$ correspond to the sum of products of weights $Y_{ij}$, and $A$ is a Bernoulli matrix representing the underlying graph structure. Then
\begin{align} \label{weight}
(W)_{ij}=\sum_{P: i\to j}Y_{i\xi_1}Y_{\xi_1 \xi_2}\cdots Y_{\xi_{k-1}j}
\end{align}
where the sum ranges over all paths $i\xi_1 \xi_2 \ldots \xi_{k-1}j$ from $i$ to $j$ with length $k$. We will bound both the number of possible paths and  the product of weights. From now we will denote $k=k(n)$ to emphasize that we will choose $k$ to be a very rapidly increasing function depending on $n$. 

The key idea is that an edge not present in any directed cycle can be traversed at most once by some path $P$. Similarly, a directed cycle can be visited at most once, although our path will typically loop multiple times at the single visit. Then any path $i\to j$ can be characterized by the order of directed cycles and the remaining edges not in cycles, along with the number of times it loops in each cycle. With high probability we may assume that $X$ has at most $2n$ edges. Then a crude bound for this quantity is
\begin{align} \label{number of paths}
(2n+\eta (n))!  k(n)^{\eta (n)} \ll (3n)!k(n)^{\log n}
\end{align}
assuming $\eta (n) \ll \log n$. 

From the inequality
\begin{align}
\mathbb{P}\left( \max_{1\leq i,j\leq n}|Y_{ij}|>n^2\right) & \leq n^2 \mathbb{P}(|Y_{11}|>n^2) \leq n^{-2}\mathbb{E}|Y_{11}|^2 =n^{-2},
\end{align}
we may condition on the high probability event where all entries have absolute value bounded by $n^2$. Now we make a stronger estimate for weights present in cycles, which will compose most of the product in \eqref{weight} for very large $k(n)$. Let $E$ denote the set of edges in the cycles.
Since the number of edges in cycles are bounded by $\eta (n)\omega (n)$, 
\begin{align} \label{strong}
\mathbb{P}\left( \max_{e\in E}|Y_e|>t(n)\right) & =1-\mathbb{P}\left( \max_{e\in E}|Y_e| \leq t(n)\right) =1-\prod_{e\in E}\mathbb{P}(|Y_e|\leq t(n))\nonumber\\
& =1-\left( 1-\mathbb{P}(|Y_{11}|>t(n))\right)^{\eta (n)\omega (n)} \nonumber\\
& \leq 1-(1-t(n)^{-2})^{\eta (n)\omega (n)}\nonumber\\
& \approx 1-e^{-\frac{\eta (n)\omega (n)}{t(n)^2}},
\end{align}
where we have used Markov's inequality in the third line.
By choosing $t(n)^2\gg \eta (n)\omega (n)$, we may condition on the high probability event where the absolute values of weights on edges in cycles are bounded by $t(n)$. For the matrix $A$, we use the rough estimate
\begin{align} \label{frobenius}
\| A\| \leq \| A\|_F \leq n^2
\end{align}
where $\|\cdot \|_F$ denotes the Frobenius norm.

We now estimate $\| Z^{k(n)}\|$. Recall that each edge not in a cycle is traversed at most once, so the contribution from weights on these edges is at most $n^{4n}$. Combining \eqref{number of paths} and \eqref{frobenius}, with high probability, we have
\begin{align}
\| Z^{k(n)}\| & = \| W\odot A\| \leq \max_{1\leq i,j\leq n}|W_{ij}|\| A\| \ll (3n)!k(n)^{\log n}n^{4n}t(n)^{k(n)}n^2 \ll n^{5n}k(n)^{\log n}(t(n))^{k(n)}
\end{align}
and
\begin{align}
\rho (Z) \leq \| Z^{k(n)}\|^{\frac{1}{k(n)}} \ll (n^{5n})^{\frac{1}{k(n)}}k(n)^{\frac{\log n}{k(n)}}t(n) \approx t(n)
\end{align}
for $k(n)= n^{n}$. Note that by choosing appropriate $\eta (n)$ and $\omega (n)$, $t(n)$ can be chosen as an arbitrary function satisfying $\lim_{n\to\infty}t(n)=\infty$, and the proof is complete.

\subsection{Proof of Theorem \ref{constant lower bound}}

By Theorem \ref{d<1 acyclic}, $\mathcal{G}_{\mathsf{d}}(n,d/n)$ for $d<1$ has a directed cycle with probability converging to $1-(1-d)e^{d+d^2/2}$. We condition on the event in which a directed cycle exists. By the reasoning in the previous subsection, with high probability, there exists a directed cycle with size smaller than $a(n)$ for some function $a$ with $\lim_{n\to\infty}a(n)=\infty$ to be determined later. We begin with the following facts:
\begin{enumerate}
\item If $i,j$ belong to the same strongly connected component $C$, then there is at most one path of length $k$ from $i$ to $j$. This path is contained completely in $C$.

\begin{proof}
Any path from $i$ to $j$ contained in $C$ is uniquely determined since $C$ is a directed cycle. Suppose a path from $i$ to $j$ contains a vertex $v\notin C$. Then there exist paths $P_1, P_2$ connecting $i\to v$ and $v\to j$, respectively. Since $C$ is a directed cycle, for any $\ell \in C$, we can choose paths $P_3: \ell \to i$ and $P_4: j \to \ell$. Then there exist paths $P_3\cup P_1: \ell \to v$ and $P_2 \cup P_4: v\to \ell$ so $v\in C$, which is a contradiction.
\end{proof} 
\item Assuming the conditions of (i), since $C$ is a directed cycle, if there exists a path from $i$ to $j$, $j$ is uniquely determined by $i$. 
\end{enumerate}
By choosing a vertex $i$ in our cycle $C$, a path $ij_1j_2\ldots j_k$ of length $k$ starting from $i$ is uniquely determined with $j_{\ell}\in C$ for $1\leq \ell \leq k$. This path is a repetition of $ij_1j_2\ldots j_{a(n)}$, which is $C$ itself, plus a remaining path which is a part of $C$.

For some function $\theta$ such that $\lim_{n\to\infty}\theta (n)=0$, with the understanding that $i=j_0$,
\begin{align}
\mathbb{P}\left( \min_{1\leq \ell \leq a(n)}|Y_{j_{\ell -1}j_{\ell}}| > \theta (n)\right) & =\left( \mathbb{P}(|Y_{ij}|>\theta (n))\right)^{a(n)}.
\end{align}
Now using the condition that $\mathbb{P}(Y_{11}=0)=0$, we may choose $a(n)$ to grow sufficiently small so that the right hand side converges to $1$. Hence we may condition on the high probability event that all edge-weights on $C$ are bounded below in magnitude by $\theta (n)$. This implies
\begin{align} \label{theta}
\| Z^k\| \geq \max_{1\leq i,j\leq n}|Z_{ij}| \geq \prod_{1\leq \ell \leq k}|Y_{j_{\ell -1}j_{\ell}}| \geq (\theta (n))^k.
\end{align} 
We also know from \cite{h24} that
\begin{align}
\mathbb{P}(\| Z\| \leq (1+\delta )\lambda_{\alpha}^{\mathrm{light}}) \approx 1-n^{1-(1+\delta )^2}
\end{align}
for $\alpha >2$ and
\begin{align*}
\lambda_{\alpha}^{\mathrm{light}} : =B_{\alpha}\frac{(\log n)^{\frac{1}{2}}}{(\log \log n)^{\frac{1}{2}-\frac{1}{\alpha}}} \ll (\log n)^{\frac{1}{2}}
\end{align*}
for some constant $B_{\alpha}$ depending only on $\alpha$.
Similarly, 
\begin{align}
\mathbb{P}(\| Z\| \leq (1+\delta )\lambda_{\alpha}^{\mathrm{heavy}}) \approx 1-n^{1-(1+\delta )^{\alpha}}
\end{align}
for $0<\alpha \leq 2$ and $\lambda_{\alpha}^{\mathrm{heavy}}:=(\log n)^{\frac{1}{\alpha}}$.
Now we establish a very rough estimate: with probability, say at least $1-(\log n)^{-1}$,
\begin{align} \label{power norm estimate}
\frac{\| Z\|^n}{\| Z^n\|} \leq \frac{\left( (\log n)^{1/\min \lbrace 2,\alpha \rbrace}\right)^n}{(\theta (n))^n} \ll (\theta (n))^{-n}(\log n)^{n^2/2} \ll (\log n)^{n^2}
\end{align}
for sufficiently slowly decreasing $\theta$.
\begin{lemma} \label{lower bound limit}
The following limit holds with $k=n^{n^n}$ and with high probability:
\begin{align} \label{product norm bound}
\lim_{n\to\infty}C_n^{-\sigma_n (k)/k}\left( \frac{\| Z\|^n}{\| Z^n\|}\right)^{-\nu_n(k)/k} =1.
\end{align}
\end{lemma}

\begin{proof}
It is clear that the limit is less than or equal to $1$. We show that it is actually $1$ using the bound \eqref{power norm estimate}. The following limits can be explicitly calculated.
\begin{gather*}
\lim_{n\to\infty}C_n^{-\sigma_n(k)/k} = \lim_{n\to\infty}\left( n^{3n/2}\right)^{-\frac{(n-1)^3}{(n-2)^2}\left( n^{n^n}\right)^{\frac{\log (n-1)}{\log n}-1}}=1,\\
\lim_{n\to\infty}\left( \frac{\| Z\|^n}{\| Z^n\|}\right)^{-\nu_n(k)/k} \geq \lim_{n\to\infty}\left((\log n)^{n^2}\right)^{-\frac{(n-1)^2}{n-2}\left( n^{n^n}\right)^{\frac{\log (n-1)}{\log n}-1}} =1.
\end{gather*}
\end{proof}

Now applying Lemma \ref{lower bound limit} and \eqref{theta} to Theorem \ref{rho bound} yields $\rho (Z) \geq \theta (n)$ with high probability.

\section{Proof of Theorem \ref{polynomial theorem}}

Here we prove Theorem \ref{polynomial theorem}.
The lower bound in this case is an immediate consequence of the sparse circular law, Theorem \ref{sparse circular law}. We will prove the upper bound using Theorem \ref{upper bound}.

Let $X_{ij}\sim \mathrm{Bernoulli}(p_n)$ and $Y_{ij}$ have a Weibull distribution with shape parameter $\alpha >0$ satisfying $\mathbb{E}Y_{ij}=0$ and $\mathbb{E}|Y_{ij}|^2=1$. We consider the i.i.d. matrix $H$ with normalized entries
\begin{align*}
H_{ij}:=\frac{1}{\sqrt{np_n}}X_{ij}Y_{ij}.
\end{align*}
Clearly $\mathbb{E}|H_{ij}|^2 \leq \frac{1}{n}$. For any constant $t>0$,
\begin{align*}
\mathbb{P}\left( \max_{1\leq i,j\leq n}|H_{ij}|  >t \right) & \leq \sum_{1\leq i,j\leq n}\mathbb{P}\left( |H_{ij}|>t\right) =n^2 \mathbb{P}(|Y_{ij}|>t\sqrt{np_n})\mathbb{P}(X_{ij}=1) \lesssim n^2 p_n e^{-(t\sqrt{np_n})^{\alpha}}.
\end{align*}
By our assumption $p_n\gg (\log n)^{\frac{2}{\alpha}}/n$ so $n^2p_n \ll e^{(np_n)^{\frac{\alpha}{2}}}$ and the probability converges to $0$. Thus we may condition on the high probability event
\begin{align*}
\mathcal{A}_q :=\left\lbrace \max_{1\leq i,j\leq n}|H_{ij}|\leq \frac{1}{q}\right\rbrace 
\end{align*}
for some appropriate $q$ depending on $c,\varepsilon$, for instance, $q=\frac{3}{c\log (1+\varepsilon )}$. Now applying Theorem \ref{upper bound} to $H$, we have
\begin{align*}
\lim_{n\to\infty}\mathbb{P}(\rho (H)\geq 1+\varepsilon )=0
\end{align*}
for any $\varepsilon >0$ and the proof is complete.

\end{document}